\documentclass[12pt,english,a4paper,oneside]{amsart}

\usepackage[cp1251]{inputenc}
\usepackage[english]{babel}
\usepackage{amsmath,amsthm,amssymb}

\usepackage[
a4paper,
left=22mm,
right=18mm,
top=20mm,
bottom=27mm
]{geometry}

\usepackage{xcolor}
\usepackage[colorlinks]{hyperref}

\newtheorem{theorem}{Theorem}[section]
\newtheorem{lemma}[theorem]{Lemma}
\newtheorem{corollary}[theorem]{Corollary}

\numberwithin{equation}{section}
\theoremstyle{definition}

\begin{document}

\title{Absolute continuity of two-dimensional polynomial random vectors}

\author{Egor Kosov}

\address{\noindent Egor Kosov,
	Centre de Recerca Matem\`atica, Campus de Bellaterra, Edifici~C 08193
	Bellaterra (Barcelona), Spain.}
\email{kosoved09@gmail.com}



\begin{abstract}

Let $X=\{X_j\}_{j=1}^\infty$ be a sequence of independent random variables
whose densities and moments of order $2d$ are uniformly bounded. For a random
vector $f(X)=(f_1(X),f_2(X))$ whose components are polynomial functionals of
degree at most $d$, we prove that
\[
[[f]]_{\mu,\infty}^{\frac1{2d-1}}\mu(f\in A)
\le
C\bigl(\lambda_2(A)\bigr)^{\frac1{2d-1}}
\]
for every Borel set $A\subset\mathbb R^2$, where $C$ depends only on $d$ and
the uniform density and moment bounds, and $\lambda_2$ denotes the Lebesgue
measure on $\mathbb R^2$.
Here
$[[f]]_{\mu,\infty}$ measures the failure of proportionality of the
highest-order orthogonal-chaos components of $f_1$ and $f_2$ with respect to
the law $\mu$ of $X$.
Consequently, whenever these
components are not
proportional, the law of $f$ admits a density in the weak Lorentz space
$L^{\frac{2d-1}{2d-2},\infty}(\mathbb R^2)$. This recovers the dichotomy
established by Nualart and Tudor for two-dimensional Wiener chaos vectors
and extends it beyond the Gaussian setting.

We also obtain the lower bound
\[
\int_{\mathbb R^\infty}\Delta_f\,d\mu
\ge
C[[f]]_{\mu,\infty}^2,
\]
where $\Delta_f$ is the determinant of the Gram matrix of $\nabla f_1$ and
$\nabla f_2$. In the special case of Gaussian measures, this gives a relaxed
version of the estimate conjectured by Nourdin, Nualart, and Poly.
\end{abstract}

\maketitle

\section{Introduction}

Let $\gamma$ be a Gaussian measure on an abstract Wiener space $E$, and let
$\mathcal H_d(\gamma)$ denote the Wiener chaos of order $d$. Nourdin, Nualart,
and Poly \cite{NNP} raised the following two questions concerning
two-dimensional random vectors
$f=(f_1,f_2)$ with $f_1,f_2\in\mathcal H_d(\gamma)$, see
Questions~6.1 and~6.2 in \cite{NNP-arxiv}.

\smallskip

\noindent
{\bf Question 1.} Is it true that the law of
$f=(f_1,f_2)$, where $f_1,f_2\in\mathcal H_d(\gamma)$, is not absolutely
continuous if and only if $f_1$ and $f_2$ are linearly dependent?

\smallskip

It was shown in \cite{NNP} that the law of
$f=(f_1,f_2)$, with $f_1,f_2\in\mathcal H_d(\gamma)$, is not absolutely
continuous if and only if the expectation of the determinant $\Delta_f$ of
its Malliavin matrix, that is, the Gram matrix of the gradients of its
components, vanishes. This led to the following quantitative counterpart of Question~1 concerning
the relation between the Malliavin determinant and the determinant of the
covariance matrix $\operatorname{Cov}_\gamma(f)$ of $f$ with respect to
$\gamma$.

\smallskip

\noindent
{\bf Question 2.}
Is it true that there exists a constant $C(d)>0$, depending only on $d$, such
that, for every $f=(f_1,f_2)$ with
$f_1,f_2\in\mathcal H_d(\gamma)$, one has
\begin{equation}\label{NNP-conj}
\int_E\Delta_f\,d\gamma
\ge
C(d)\det\operatorname{Cov}_{\gamma}(f)\quad?
\end{equation}

\smallskip

We note that both questions were originally formulated for general
$k$-dimensional random vectors with components in a fixed Wiener chaos, but
simple examples show that both have negative answers when $k>2$, see \cite{NT17}.

For $d=2$, both questions were answered affirmatively in \cite{NNP}. The first
question was subsequently answered affirmatively for $d=3$ and $d=4$ by Tudor
\cite{Tudor13}. Nualart and Tudor \cite{NT17} proved \eqref{NNP-conj} for
$d=3$ and $d=4$. For $d\ge5$, they established a weaker substitute involving
additional nonnegative terms:
\begin{equation}\label{eq-NT}
\int_E\Delta_f\,d\gamma
+
\sum_{j=2}^{\lfloor (d-1)/2\rfloor}
C_j(d)\int_E\Delta_f^{(j)}\,d\gamma
\ge
C(d)\det\operatorname{Cov}_{\gamma}(f),
\end{equation}
where $\Delta_f^{(j)}$ denotes the determinant of the $j$th iterated
Malliavin matrix. In particular, this allowed Nualart and Tudor to provide an
affirmative answer to Question~1 for every $d\ge2$.

In this paper, we show that the phenomenon described in Question~1 and
established by Nualart and Tudor in the Gaussian setting is not specific to
Gaussian measures, but is instead a consequence of the polynomial nature of
the two-dimensional random vector and the independence of the coordinates of
the underlying random element. 
More precisely, we prove that the same dichotomy remains valid for
two-dimensional random vectors whose components belong to the orthogonal
polynomial chaos of a fixed order $d$ associated with a random element whose
coordinates are independent, have uniformly bounded densities, and satisfy a
suitable uniform moment assumption.

In the same general setting, we also establish a relaxed version of
\eqref{NNP-conj}, replacing the determinant of the covariance matrix by a
smaller coefficient-based parameter. Although the resulting estimate is
generally weaker than the conjectured Gaussian estimate, it holds in a broader
setting, and the parameter vanishes if and only if the covariance determinant
vanishes, thereby detecting exactly the same degeneracy of the distribution.

\subsection{Main results}
For $M,R>0$ and $d\in\mathbb N$, let $\mathcal{M}_d(M,R)$ denote the
family of all Borel probability measures $\nu$ on $\mathbb R$ that are
absolutely continuous and admit a density $\varrho$ satisfying
$\varrho(t)\le M$ for all $t\in\mathbb R$ and
\[
\int_{\mathbb R}|t|^{2d}\,\nu(dt)\le R^{2d}.
\]
We consider product probability measures
$\mu=\bigotimes_{j=1}^\infty\mu_j$ on $\mathbb R^\infty$ with
$\mu_j\in\mathcal M_d(M,R)$. In other words, $\mu$ is the joint law of a
sequence of independent random variables whose densities are bounded by $M$
and whose moments of order $2d$ do not exceed $R^{2d}$. We note that
every isotropic log-concave probability measure $\nu$ on $\mathbb R$ belongs
to $\mathcal M_d(1,4d)$ by
\cite[Lemma~5.5 and Theorem~5.22]{LV07}. In particular, the standard Gaussian
measure $\gamma_1$ is isotropic and log-concave. Moreover, by Tsirelson's
isomorphism theorem, every centered Radon Gaussian measure with
infinite-dimensional Cameron--Martin space is measurably linearly isomorphic
to
$\gamma:=\gamma_1^{\otimes\infty}$, the countable product of copies of $\gamma_1$, see
\cite[Theorem~3.4.4]{Gaus}. Consequently, the results below apply, in
particular, to products of isotropic log-concave measures and to the abstract
Wiener space setting.

We call a function on $\mathbb R^\infty$
cylindrical if it depends only on finitely many coordinates. Let
$\mathcal P_d(\mu)$ be the closure in $L^2(\mu)$ of the space of all
cylindrical polynomials of degree at most $d$. Let $\mathcal H_d(\mu)$ be the
orthogonal complement of $\mathcal P_{d-1}(\mu)$ in $\mathcal P_d(\mu)$.

Let $\{u_{j,m}\}_{m=0}^d$, $u_{j,0}=1$, be orthonormal polynomials in $L^2(\mu_j)$ such that
$u_{j,m}$ has degree $m$. Then the system
\[
U_{\mathbf m}(x):=\prod_{j=1}^\infty u_{j,m_j}(x_j),
\qquad
\mathbf m=(m_j)_{j=1}^\infty\in\mathbb Z_+^\infty,
\qquad
|\mathbf m|:=\sum_{j=1}^\infty m_j\le d,
\]
forms an orthonormal basis of $\mathcal P_d(\mu)$.

Let $f=(f_1,f_2)$ with
$f_1,f_2\in\mathcal P_d(\mu)$. The components admit the following expansions
in this basis, converging in $L^2(\mu)$:
\[
f_1(x)=\sum_{|\mathbf m|\le d}a_{\bf m}^\mu(f_1)U_{\mathbf m}(x),
\qquad
f_2(x)=\sum_{|\mathbf m|\le d}a_{\bf m}^\mu(f_2)U_{\mathbf m}(x).
\]
For nonconstant $f$, let
\[
d_\mu(f):=
\max\{|\mathbf m|\colon |a_{\bf m}^\mu(f_1)|+|a_{\bf m}^\mu(f_2)|\ne0\}.
\]
For $\mathbf p,\mathbf q$ with $|\mathbf p|=|\mathbf q|=d_\mu(f)$, let
\[
A^\mu_{\mathbf p,\mathbf q}(f)
=
a_{\bf p}^\mu(f_1)a_{\bf q}^\mu(f_2)
-
a_{\bf q}^\mu(f_1)a_{\bf p}^\mu(f_2)
=
\det
\begin{pmatrix}
	a_{\bf p}^\mu(f_1) & a_{\bf q}^\mu(f_1)\\
	a_{\bf p}^\mu(f_2) & a_{\bf q}^\mu(f_2)
\end{pmatrix},
\]
and set
\[
[[f]]_{\mu,2}
:=
\biggl(
\frac12
\sum_{|\mathbf p|=|\mathbf q|=d_\mu(f)}
|A^\mu_{\mathbf p,\mathbf q}(f)|^2
\biggr)^{1/2}
\quad
\text{and}
\quad
[[f]]_{\mu,\infty}
:=
\sup_{|\mathbf p|=|\mathbf q|=d_\mu(f)}
|A^\mu_{\mathbf p,\mathbf q}(f)|.
\]
For constant $f$, we set
$d_\mu(f)=0$ and
$[[f]]_{\mu,2}=[[f]]_{\mu,\infty}=0$.

We note that, for $f=(f_1,f_2)$ with $f_1,f_2\in\mathcal H_d(\mu)$, the
Cauchy--Binet formula gives
\[
[[f]]_{\mu,2}^2=\det\operatorname{Cov}_\mu(f),
\]
where $\operatorname{Cov}_\mu(f)$ denotes the covariance matrix of $f$ with
respect to $\mu$. Thus, for such random vectors $f$, the quantity
$[[f]]_{\mu,\infty}$ is the $\ell^\infty$ counterpart of
$\sqrt{\det\operatorname{Cov}_\mu(f)}$. In particular,
\begin{equation}\label{eq-equiv}
[[f]]_{\mu,\infty}=0
\Longleftrightarrow
\det\operatorname{Cov}_\mu(f)=0.
\end{equation}

Our first main result provides the following quantitative version of the
dichotomy in Question~1 beyond the Gaussian setting.

\begin{theorem}\label{th-inf-lorentz}
Let $M,R>0$ and let $d\in\mathbb N$, $d\ge2$.
There exists a constant $C(M,R,d)>0$ such that, for every
$\mu=\bigotimes_{j=1}^{\infty}\mu_j$
with $\mu_j\in\mathcal M_d(M,R)$ and every $f=(f_1,f_2)$ with
$f_1,f_2\in\mathcal P_d(\mu)$, one has
\[
[[f]]_{\mu,\infty}^{\frac1{2d-1}}\mu(f\in A)
\le
C(M,R,d)\bigl(\lambda_2(A)\bigr)^{\frac1{2d-1}}
\]
for every Borel set $A\subset\mathbb R^2$.
\end{theorem}

Taking $A=\{y\in\mathbb R^2\colon |y|\le\varepsilon\}$ in
Theorem~\ref{th-inf-lorentz}, we obtain the following small-ball estimate of
Carbery--Wright type \cite{CW01}:
\[
[[f]]_{\mu,\infty}^{\frac1{2d-1}}\mu(|f|\le\varepsilon)
\le
C(M,R,d)\varepsilon^{\frac2{2d-1}}
\quad\text{for every }\varepsilon>0.
\]

Equivalently, Theorem~\ref{th-inf-lorentz} states that, if
$f=(f_1,f_2)$ with $f_1,f_2\in\mathcal P_d(\mu)$ and
$[[f]]_{\mu,\infty}>0$, then the law of $f$ admits a density
$\varrho_f$ in the weak Lorentz space
$L^{\frac{2d-1}{2d-2},\infty}(\mathbb R^2)$ satisfying
\[
\|\varrho_f\|_{L^{\frac{2d-1}{2d-2},\infty}(\mathbb R^2)}
\le
C(M,R,d)[[f]]_{\mu,\infty}^{-\frac1{2d-1}},
\]
where the standard weak-Lorentz quasi-norm is defined by
\[
\|\varrho_f\|_{L^{p,\infty}(\mathbb R^2)}
:=
\sup_{s>0}
s\bigl(\lambda_2(\varrho_f\ge s)\bigr)^{1/p}.
\]
Since $\varrho_f$ is a probability density, it follows that
$\varrho_f\in L^q(\mathbb R^2)$
for every $1\le q<\frac{2d-1}{2d-2}$.

In particular, by \eqref{eq-equiv}, for
$f=(f_1,f_2)$ with $f_1,f_2\in\mathcal H_d(\mu)$, the components are either
proportional or else the law of $f$ admits a density in
$L^{\frac{2d-1}{2d-2},\infty}(\mathbb R^2)$.

In the log-concave setting, the weak-Lorentz regularity of densities of
polynomial images was studied in \cite{KosZh} under an algebraic
nondegeneracy condition excluding polynomial relations of degree at most
$d$ between the components. By contrast, in the product setting considered
here, Theorem~\ref{th-inf-lorentz} requires only the exclusion of a linear
relation between the highest-order chaos components of $f_1$ and $f_2$.

\medskip

Our second main result is a relaxed version of the conjectured estimate
\eqref{NNP-conj} beyond the Gaussian setting.
By approximation in $L^2(\mu)$, the gradient extends uniquely from
cylindrical polynomials to $\mathcal P_d(\mu)$, see Corollary~\ref{cor-grad}. For
$f=(f_1,f_2)$ with $f_1,f_2\in\mathcal P_d(\mu)$, we set
\[
\Delta_f
:=
|\nabla f_1|^2|\nabla f_2|^2
-
\langle\nabla f_1,\nabla f_2\rangle^2
=
\det
\begin{pmatrix}
|\nabla f_1|^2 & \langle\nabla f_1,\nabla f_2\rangle\\
\langle\nabla f_1,\nabla f_2\rangle & |\nabla f_2|^2
\end{pmatrix}.
\]

\begin{theorem}\label{th-det-est}
Let $M,R>0$ and let $d\in\mathbb N$, $d\ge2$.
There exists a constant $C(M,R,d)>0$ such that, for every
$\mu=\bigotimes_{j=1}^{\infty}\mu_j$
with $\mu_j\in\mathcal M_d(M,R)$ and every $f=(f_1,f_2)$ with
$f_1,f_2\in\mathcal P_d(\mu)$, one has
\[
\int_{\mathbb R^\infty}\Delta_f\,d\mu
\ge
C(M,R,d)[[f]]_{\mu,\infty}^2.
\]
\end{theorem}

In the Gaussian setting, the estimate of Nualart and Tudor
\eqref{eq-NT} involves additional determinants of iterated Malliavin matrices
when $d\ge5$. In contrast, Theorem~\ref{th-det-est} applies in a more general
setting and, for $f=(f_1,f_2)$ with $f_1,f_2\in\mathcal H_d(\mu)$, provides a
direct lower bound for the expectation of $\Delta_f$ itself. The price is that
$\det\operatorname{Cov}_\mu(f)=[[f]]_{\mu,2}^2$
is replaced by the smaller coefficient-based quantity
$[[f]]_{\mu,\infty}^2$.

\subsection{Main ingredients of the proofs}
At the core of the proofs of
Theorems~\ref{th-inf-lorentz} and~\ref{th-det-est} lies a new
dimension-free algebraic estimate.

Let $\mathcal P_d(\mathbb R^n;\mathbb R^k)$
denote the space of all mappings
$f=(f_1,\ldots,f_k)\colon \mathbb R^n\to\mathbb R^k$ of the form
\[
f(x)=\sum_{|\mathbf m|\le d} a_{\bf m}(f)x^{\mathbf m},
\]
where $x=(x_1,\ldots,x_n)$,
$\mathbf m=(m_1,\ldots,m_n)\in\mathbb Z_+^n$,
$|\mathbf m|=m_1+\ldots+m_n$,
$x^{\mathbf m}=x_1^{m_1}\ldots x_n^{m_n}$, and
\[
a_{\bf m}(f)=
\bigl(a_{\bf m}(f_1),\ldots,a_{\bf m}(f_k)\bigr)\in\mathbb R^k.
\]
For nonzero $f$, we set
\[
d(f):=\max\{|\mathbf m|\colon a_{\bf m}(f)\ne 0\},
\]
and for the zero polynomial we set $d(0)=0$.

When $k=1$, we write $\mathcal P_d(\mathbb R^n)$ instead of
$\mathcal P_d(\mathbb R^n;\mathbb R)$. For nonzero
$f\in\mathcal P_d(\mathbb R^n)$, let
\[
[f]_\infty:=\max_{|\mathbf m|=d(f)}|a_{\bf m}(f)|,
\]
and set $[0]_\infty=0$.

When $k=2$ and $d(f)\ge1$, for multi-indices $\mathbf p,\mathbf q$ with
$|\mathbf p|=|\mathbf q|=d(f)$, let
\[
A_{\mathbf p,\mathbf q}(f)
=
a_{\mathbf p}(f_1)a_{\mathbf q}(f_2)
-
a_{\mathbf q}(f_1)a_{\mathbf p}(f_2)
=
\det
\begin{pmatrix}
a_{\mathbf p}(f_1) & a_{\mathbf q}(f_1)\\
a_{\mathbf p}(f_2) & a_{\mathbf q}(f_2)
\end{pmatrix}
\]
and let
\[
[[f]]_\infty:=
\max_{|\mathbf p|=|\mathbf q|=d(f)}
|A_{\mathbf p,\mathbf q}(f)|.
\]
For constant $f$, we set $[[f]]_\infty=0$.

For $1\le k\le n$, set
\[
{\bf N}^k_n
:=
\{{\bf r}=(r_1,\ldots,r_k)\in\mathbb N^k\colon
1\le r_1<\ldots<r_k\le n\}.
\]
For fixed ${\bf r}\in{\bf N}^k_n$, we define the partial Jacobian matrix by
\[
J_f^{\bf r}
:=
\bigl(\partial_{x_{r_j}} f_i\bigr)_
{\substack{1\le i\le k\\ 1\le j\le k}}.
\]
By the Cauchy--Binet formula, for the Jacobian matrix $J_f$, we have
\begin{equation}\label{eq-decomp}
\Delta_f:=\det(J_fJ_f^*)
=
\sum_{{\bf r}\in{\bf N}_n^k}|\det J_f^{\bf r}|^2.
\end{equation}
Thus, for vector-valued mappings, the quantities $\det J_f^{\bf r}$ play
the role of partial derivatives.

\begin{theorem}\label{th-algebr}
Let $d\in\mathbb N$. There exists a constant $C(d)>0$ such that, for every $n\in\mathbb N$, $n\ge2$, and every nonconstant
$f=(f_1,f_2)\in\mathcal P_d(\mathbb R^n;\mathbb R^2)$, one has
\[
\max_{{\bf r}\in{\bf N}_n^2}[\det J_f^{\bf r}]_\infty
\ge
C(d)[[f]]_\infty.
\]
\end{theorem}

Theorem~\ref{th-algebr} enters the proof of
Theorem~\ref{th-det-est} through the decomposition \eqref{eq-decomp}.
Indeed, an integral norm of a polynomial can often be bounded below in terms
of the coefficients of its leading homogeneous part. For example, the argument of Glazer and Mikulincer \cite{GM22} yields that,
if each $\mu_j$ is an isotropic log-concave probability measure and
$g\in\mathcal P_d(\mathbb R^n)$, then
\[
\int_{\mathbb R^n}|g|^2\,d\mu_1\otimes\cdots\otimes\mu_n
\ge
C(d)[g]_2^2,
\quad\text{where}\quad
[g]_2
:=
\biggl(
\sum_{|\mathbf m|=d(g)}
|a_{\mathbf m}(g)|^2
\biggr)^{1/2}.
\]
We show in Corollary~\ref{cor-GM-type} that, for the more general class of
product measures whose factors have densities bounded by $M$, the following estimate holds:
\[
\int_{\mathbb R^n}|g|\,d\mu_1\otimes\cdots\otimes\mu_n
\ge
C(M,d)[g]_\infty.
\]
Applying this estimate with degree $2d-2$ to the polynomials
$\det J_f^{\mathbf r}$ and combining it with \eqref{eq-decomp} and
Theorem~\ref{th-algebr}, we obtain the finite-dimensional
version of Theorem~\ref{th-det-est}:
\[
\int_{\mathbb R^n}\Delta_f\,d\mu_1\otimes\cdots\otimes\mu_n
\ge
C(M,d)[[f]]_\infty^2
\]
for every $f\in\mathcal P_d(\mathbb R^n;\mathbb R^2)$.

\medskip

Theorem~\ref{th-algebr} enters the proof of
Theorem~\ref{th-inf-lorentz} through the following general Lorentz-type
regularity estimate for polynomial random vectors generated by independent
random variables with uniformly bounded densities.

\begin{theorem}\label{th-lorentz}
Let $a,M>0$ and let $k,d,n\in\mathbb N$ with $n\ge k$ and $d\ge2$. Let
$\mu=\mu_1\otimes\cdots\otimes\mu_n$,
where each $\mu_j$ is a probability measure with density $\varrho_j$ satisfying
\[
0\le \varrho_j(x_j)\le M
\quad \forall x_j\in\mathbb R,
\quad \forall j\in\{1,\ldots,n\}.
\]
Let
$f\in\mathcal P_d(\mathbb R^n;\mathbb R^k)$
satisfy
\[
\max_{{\bf r}\in{\bf N}_n^k}[\det J_f^{\bf r}]_\infty\ge a.
\]
Then
\[
\mu(f\in A)
\le
20k^2d^2(1+M)^2a^{-\frac{1}{k(d-1)+1}}
\bigl(\lambda_k(A)\bigr)^{\frac{1}{k(d-1)+1}}
\]
for every Borel set $A\subset\mathbb R^k$.
\end{theorem}

This theorem can be viewed as a polynomial analogue of the result of
Rudelson and Vershynin \cite{RV15} for linear mappings
(see also \cite{LPP16,MMX17} for further developments). More precisely, for the same class of product measures with uniformly bounded
marginal densities,
their result implies that, for a linear mapping $f(x)=Lx$,
\[
\mu(f\in A)
\le
(CM)^k a^{-1}\lambda_k(A)
\]
for every Borel set $A\subset\mathbb R^k$, provided that
\[
\det(J_fJ_f^*)=\sum_{{\bf r}\in{\bf N}_n^k}|\det J_f^{\bf r}|^2
\ge a^2.
\]
We use Theorem~\ref{th-lorentz} together with the algebraic estimate in
Theorem~\ref{th-algebr} to obtain a finite-dimensional version of
Theorem~\ref{th-inf-lorentz}.

\subsection{Notation}

Let $C_0^\infty(\mathbb R^n)$ denote the space of all smooth compactly
supported functions on $\mathbb R^n$. Let $\langle x,y\rangle$ denote the
standard inner product on both $\mathbb R^n$ and $\ell^2$, and let $|x|$
denote the corresponding norm. 
Let $\lambda_n$ denote the Lebesgue measure on $\mathbb R^n$.
For a measurable mapping
$g\colon\mathbb R^\infty\to\mathbb R^n$ or
$g\colon\mathbb R^\infty\to\ell^2$, we use the notation
\[
\|g\|_{L^p(\mu)}:=\||g|\|_{L^p(\mu)}\quad 
\text{for}\quad p\in[1,\infty).
\]

Throughout the paper, constants are denoted by $C$ and may change from line
to line. Their dependence on parameters is always indicated explicitly, for
instance by writing $C(M,R,d)$. When several constants appear in the same
argument, we distinguish them by subscripts, writing $C_1,C_2$, and so on.
These subscripts are only labels and do not indicate any additional
dependence.

\subsection{Structure of the paper}

The rest of the paper is organized as follows. In
Section~\ref{sect-algebr}, we prove the key algebraic estimate stated in
Theorem~\ref{th-algebr}. Section~\ref{sect-lorentz} is devoted to the proof
of the general Lorentz-type regularity theorem, Theorem~\ref{th-lorentz},
and auxiliary results. Finally, in Section~\ref{sect-main}, we prove our main
results, Theorems~\ref{th-inf-lorentz} and~\ref{th-det-est}.

\section{Algebraic estimate}
\label{sect-algebr}

\subsection{A dimension-dependent estimate for two homogeneous polynomials}

\begin{lemma}\label{lem-homogeneous-dimensional}
Let $n,d\in\mathbb N$ and $n\ge2$. Then there exists a constant
$C(d,n)>0$ such that, for every $u=(u_1,u_2)$, where
\[
u_1(x_1,\ldots,x_n)
=
\sum_{|{\bf m}|=d}a_{\bf m}(u_1)x^{{\bf m}}
\quad\text{and}\quad
u_2(x_1,\ldots,x_n)
=
\sum_{|{\bf m}|=d}a_{\bf m}(u_2)x^{{\bf m}},
\]
one has
\[
\max_{{\bf r}\in {\bf N}_n^2}
[\det J_u^{\bf r}]_\infty
\ge
C(d,n)[[u]]_\infty.
\]
\end{lemma}

\begin{proof}
We recall that
\[
[[u]]_\infty
=
\max_{|{\bf p}|=|{\bf q}|=d}
|A_{{\bf p},{\bf q}}(u)|,
\]
where
\[
A_{{\bf p},{\bf q}}(u)
:=
a_{\bf p}(u_1)a_{\bf q}(u_2)
-
a_{\bf q}(u_1)a_{\bf p}(u_2).
\]
Clearly,
\[
A_{{\bf p},{\bf q}}(u)=-A_{{\bf q},{\bf p}}(u),
\quad
A_{{\bf p},{\bf p}}(u)=0.
\]

For ${\bf r}=(r_1,r_2)\in{\bf N}_n^2$, we have
\begin{align*}
\det J_u^{\bf r}
&=
\partial_{x_{r_1}}u_1\,\partial_{x_{r_2}}u_2
-
\partial_{x_{r_2}}u_1\,\partial_{x_{r_1}}u_2
\\
&=
\sum_{|{\bf p}|=|{\bf q}|=d}
(p_{r_1}q_{r_2}-p_{r_2}q_{r_1})
a_{\bf p}(u_1)a_{\bf q}(u_2)
x^{{\bf p}+{\bf q}-e_{r_1}-e_{r_2}}
\\
&=
\frac{1}{2}
\sum_{|{\bf p}|=|{\bf q}|=d}
(p_{r_1}q_{r_2}-p_{r_2}q_{r_1})
A_{{\bf p},{\bf q}}(u)
x^{{\bf p}+{\bf q}-e_{r_1}-e_{r_2}}
\\
&=
\sum_{|{\bf m}|=2d-2}
\biggl(
\sum_{\substack{|{\bf p}|=|{\bf q}|=d\\
{\bf p}+{\bf q}-e_{r_1}-e_{r_2}={\bf m}}}
\frac{1}{2}(p_{r_1}q_{r_2}-p_{r_2}q_{r_1})
A_{{\bf p},{\bf q}}(u)
\biggr)x^{\bf m}.
\end{align*}
Thus, the coefficients of all polynomials $\det J_u^{\bf r}$ depend linearly
on the family
\[
\bigl(A_{{\bf p},{\bf q}}(u)\bigr)_{|{\bf p}|=|{\bf q}|=d}.
\]

We claim that if
\[
\det J_u^{\bf r}\equiv0
\quad
\forall\,{\bf r}\in{\bf N}_n^2,
\]
then
\[
A_{{\bf p},{\bf q}}(u)=0
\quad
\forall ({\bf p}, {\bf q}),\quad |{\bf p}|=|{\bf q}|=d.
\]
Indeed, if $u_2\equiv0$, then all coefficients $a_{\bf m}(u_2)$ vanish, and the
conclusion is immediate. Assume now that $u_2\not\equiv0$. It is readily verified that
\[
\sum_{j=1}^n x_j\partial_{x_j}u_1=\langle \nabla u_1(x),x\rangle=du_1,
\quad
\sum_{j=1}^n x_j\partial_{x_j}u_2=\langle \nabla u_2(x),x\rangle=du_2.
\]
Fix $s\in\{1,\ldots,n\}$. Since all minors $\det J_u^{\bf r}$ vanish
identically, we have
\begin{align*}
0
\equiv
\sum_{j<s}x_j\det J_u^{(j,s)}
-
\sum_{j>s}x_j\det J_u^{(s,j)}
&=
\sum_{j=1}^n
x_j
\bigl(
\partial_{x_j}u_1\,\partial_{x_s}u_2
-
\partial_{x_s}u_1\,\partial_{x_j}u_2
\bigr)
\\
&=
du_1\partial_{x_s}u_2
-
du_2\partial_{x_s}u_1.
\end{align*}
Hence
\[
u_1\partial_{x_s}u_2-u_2\partial_{x_s}u_1\equiv0
\quad
\forall s\in\{1,\ldots,n\}.
\]
Therefore, on the open set $\{u_2\ne0\}$,
\[
\nabla\Bigl(\frac{u_1}{u_2}\Bigr)\equiv0.
\]
Since $u_2$ is not identically zero, the set $\{u_2\ne0\}$ contains a
nonempty open ball $B$. On this ball,
\[
u_1=cu_2
\]
for some constant $c$. Since $u_1-cu_2$ is a polynomial and vanishes on a
nonempty open ball, it vanishes identically. Thus
\[
a_{\bf m}(u_1)=c\,a_{\bf m}(u_2)
\quad
\forall {\bf m},\quad |{\bf m}|=d,
\]
and consequently
\[
A_{{\bf p},{\bf q}}(u)=0
\quad \forall ({\bf p}, {\bf q}),\quad |{\bf p}|=|{\bf q}|=d.
\]
This proves the claim.

Now consider the set
\[
S
:=
\biggl\{
\bigl(A_{{\bf p},{\bf q}}(u)\bigr)_{|{\bf p}|=|{\bf q}|=d}
\colon
\frac{1}{2}
\sum_{|{\bf p}|=|{\bf q}|=d}
|A_{{\bf p},{\bf q}}(u)|^2=1
\biggr\}.
\]

We show that $S$ is compact. By the Cauchy--Binet formula,
\[
\frac{1}{2}
\sum_{|{\bf p}|=|{\bf q}|=d}
|A_{{\bf p},{\bf q}}(u)|^2
=
|a_1|^2|a_2|^2-\langle a_1,a_2\rangle^2,
\]
where
\[
a_i=\bigl(a_{\bf m}(u_i)\bigr)_{|{\bf m}|=d},
\quad i\in\{1,2\}.
\]
Replacing $a_1$ by $a_1+\tau a_2$ does not change any minor
$A_{{\bf p},{\bf q}}(u)$. Thus every nonzero family of minors can be
represented by two orthogonal coefficient vectors. If the family belongs to
$S$, then these vectors can be rescaled by reciprocal factors so that
\[
|a_1|=|a_2|=1,
\quad
\langle a_1,a_2\rangle=0.
\]
Hence $S$ is the image of the compact set
\[
\{(a_1,a_2)\colon |a_1|=|a_2|=1,\ \langle a_1,a_2\rangle=0\}
\]
under the continuous map
\[
(a_1,a_2)\mapsto
\bigl(A_{{\bf p},{\bf q}}(u)\bigr)_{|{\bf p}|=|{\bf q}|=d}.
\]
Therefore, $S$ is compact.

Since the coefficients of each $\det J_u^{\bf r}$ are obtained by linear
transformations from the family
\[
\bigl(A_{{\bf p},{\bf q}}(u)\bigr)_{|{\bf p}|=|{\bf q}|=d},
\]
the function
\[
\bigl(A_{{\bf p},{\bf q}}(u)\bigr)_{|{\bf p}|=|{\bf q}|=d}
\mapsto
\max_{{\bf r}\in{\bf N}_n^2}[\det J_u^{\bf r}]_\infty
\]
is continuous and nonnegative on $S$.
By the claim proved above, this function
does not vanish on $S$. Consequently,
\[
C(d,n)
:=
\min_{S}
\max_{{\bf r}\in{\bf N}_n^2}[\det J_u^{\bf r}]_\infty
>0.
\]

Finally, by homogeneity,
\[
\max_{{\bf r}\in{\bf N}_n^2}[\det J_u^{\bf r}]_\infty
\ge
C(d,n)
\Bigl(\frac{1}{2}
\sum_{|{\bf p}|=|{\bf q}|=d}
|A_{{\bf p},{\bf q}}(u)|^2\Bigr)^{1/2}
\ge 
C(d,n)[[u]]_\infty
\]
for arbitrary homogeneous polynomial mappings $u$. This gives the claimed estimate.
\end{proof}

\subsection{Proof of Theorem~\ref{th-algebr}}

If $[[f]]_\infty=0$, there is nothing to prove. Assume that $[[f]]_\infty>0$
and set
$m:=d(f)\ge1$.
Let
\[
f_1=\sum_{j=0}^{m} f_1^{(j)}
\quad\text{and}\quad
f_2=\sum_{j=0}^{m} f_2^{(j)}
\]
be the decompositions of $f_1$ and $f_2$ into homogeneous components, where
$f_1^{(j)}$ and $f_2^{(j)}$ are homogeneous of degree $j$. Then
\[
\det J_f^{\bf r}
=
\sum_{s=1}^{m}\sum_{j=1}^{m}
\Bigl(
\partial_{x_{r_1}}f_1^{(s)}\,\partial_{x_{r_2}}f_2^{(j)}
-
\partial_{x_{r_2}}f_1^{(s)}\,\partial_{x_{r_1}}f_2^{(j)}
\Bigr).
\]
The summand with indices $(s,j)$ is homogeneous of degree $s+j-2$. Hence the
homogeneous component of $\det J_f^{\bf r}$ of degree $2m-2$ is
$\det J_{f^{(m)}}^{\bf r}$,
where
$f^{(m)}:=\bigl(f_1^{(m)},f_2^{(m)}\bigr)$.

Write
\[
f_1^{(m)}(x)=\sum_{|{\bf m}|=m}a_{\bf m}(f_1^{(m)})x^{{\bf m}}
\quad\text{and}\quad
f_2^{(m)}(x)=\sum_{|{\bf m}|=m}a_{\bf m}(f_2^{(m)})x^{{\bf m}}.
\]
Choose multiindices ${\bf p},{\bf q}$, $|{\bf p}|=|{\bf q}|=m$, such that
\[
\bigl|
a_{\bf p}(f_1^{(m)})a_{\bf q}(f_2^{(m)})
-
a_{\bf q}(f_1^{(m)})a_{\bf p}(f_2^{(m)})
\bigr|
=
[[f]]_\infty.
\]
Let
\[
I=\{j_1,\ldots,j_N\}:=\operatorname{supp}{\bf p}\cup\operatorname{supp}{\bf q}.
\]
Since the chosen minor is nonzero, we have $N\ge2$. Moreover,
\[
N=|I|\le 2m\le 2d.
\]
For $x_I=(x_{j_1},\ldots,x_{j_N})\in\mathbb R^N$, define
\[
u_1(x_I)=
\sum_{\substack{|{\bf m}|=m\\ \operatorname{supp}{\bf m}\subset I}}
a_{\bf m}(f_1^{(m)})x_{j_1}^{m_{j_1}}\ldots x_{j_N}^{m_{j_N}}
\quad\text{and}\quad
u_2(x_I)=
\sum_{\substack{|{\bf m}|=m\\ \operatorname{supp}{\bf m}\subset I}}
a_{\bf m}(f_2^{(m)})x_{j_1}^{m_{j_1}}\ldots x_{j_N}^{m_{j_N}}.
\]
Then $u=(u_1,u_2)$ is homogeneous of degree $m$. Since
\[
\operatorname{supp}{\bf p}\subset I
\quad\text{and}\quad
\operatorname{supp}{\bf q}\subset I,
\]
the chosen maximal minor is retained in the coefficient family of $u$.
Therefore,
\begin{equation}\label{eq-U-f}
[[u]]_\infty\ge [[f]]_\infty.
\end{equation}
By Lemma~\ref{lem-homogeneous-dimensional}, applied in dimension $N$, we obtain
\begin{equation}\label{eq-lem-app}
\max_{1\le r_1<r_2\le N}
\bigl[\det J^{(j_{r_1}j_{r_2})}_{u}\bigr]_\infty
\ge
C(d)[[u]]_\infty,
\end{equation}
where
\[
C(d):=\min_{\substack{1\le m\le d\\ 2\le N\le 2d}} C(m,N)>0.
\]

Let $v=(v_1,v_2)$, where
\[
v_1(x):=
\sum_{\substack{|{\bf m}|=m\\ \exists \ell\notin I\colon m_\ell\ne 0}}
a_{\bf m}(f_1^{(m)})x^{{\bf m}}
\quad\text{and}\quad
v_2(x):=
\sum_{\substack{|{\bf m}|=m\\ \exists \ell\notin I\colon m_\ell\ne 0}}
a_{\bf m}(f_2^{(m)})x^{{\bf m}}.
\]
Then
\[
f_1^{(m)}(x)=u_1(x_I)+v_1(x)
\quad\text{and}\quad
f_2^{(m)}(x)=u_2(x_I)+v_2(x).
\]
For fixed $1\le r_1<r_2\le N$, viewing $u_1,u_2$ as polynomials on
$\mathbb R^n$ depending only on the variables indexed by $I$, we have
\begin{align*}
\det J^{(j_{r_1},j_{r_2})}_{f^{(m)}}
&=
\det J^{(j_{r_1}j_{r_2})}_{u}
+
\det J^{(j_{r_1},j_{r_2})}_v
\\
&\quad
+
\partial_{x_{j_{r_1}}}v_1\,\partial_{x_{j_{r_2}}}u_2
-
\partial_{x_{j_{r_2}}}v_1\,\partial_{x_{j_{r_1}}}u_2
\\
&\quad
+
\partial_{x_{j_{r_1}}}u_1\,\partial_{x_{j_{r_2}}}v_2
-
\partial_{x_{j_{r_2}}}u_1\,\partial_{x_{j_{r_1}}}v_2.
\end{align*}
Every monomial of $v_1$ and $v_2$ contains at least one variable $x_\ell$ with
$\ell\notin I$. Differentiation with respect to $x_{j_{r_1}}$ or
$x_{j_{r_2}}$ cannot remove such a variable, because $j_{r_1},j_{r_2}\in I$.
Hence every nonzero monomial in each term involving $v_1$ or $v_2$ still
contains at least one variable outside $I$. Consequently, the monomials of
$\det J^{(j_{r_1},j_{r_2})}_{f^{(m)}}$ supported entirely in the variables
indexed by $I$ come only from the term $\det J^{(j_{r_1}, j_{r_2})}_u$. Therefore, the
coefficient vector of $\det J^{(j_{r_1},j_{r_2})}_u$ is a subvector of the coefficient
vector of $\det J^{(j_{r_1},j_{r_2})}_{f^{(m)}}$, and so
\[
\bigl[\det J^{(j_{r_1},j_{r_2})}_{f^{(m)}}\bigr]_\infty
\ge
\bigl[\det J^{(j_{r_1},j_{r_2})}_{u}\bigr]_\infty.
\]

By \eqref{eq-lem-app}, we choose $1\le r_1<r_2\le N$ such that
\[
\bigl[\det J_u^{(j_{r_1},j_{r_2})}\bigr]_\infty
\ge
C(d)[[u]]_\infty.
\]
Then
\[
\bigl[\det J^{(j_{r_1},j_{r_2})}_{f^{(m)}}\bigr]_\infty
\ge
C(d)[[u]]_\infty.
\]
In particular,
\[
\det J^{(j_{r_1},j_{r_2})}_{f^{(m)}}\not\equiv0.
\]
Since $\det J^{(j_{r_1},j_{r_2})}_{f^{(m)}}$ is the homogeneous component of
$\det J_f^{(j_{r_1},j_{r_2})}$ of degree $2m-2$, the polynomial
$\det J_f^{(j_{r_1},j_{r_2})}$ has degree $2m-2$. Therefore, by the definition
of $[\,\cdot\,]_\infty$,
\[
\bigl[\det J_f^{(j_{r_1},j_{r_2})}\bigr]_\infty
=
\bigl[\det J^{(j_{r_1},j_{r_2})}_{f^{(m)}}\bigr]_\infty.
\]
Taking into account \eqref{eq-U-f}, we obtain
\[
\max_{{\bf r}\in {\bf N}_n^2}[\det J_f^{\bf r}]_\infty
\ge
C(d)[[f]]_\infty,
\]
which completes the proof.
\qed

\section{Lorentz-type regularity estimates}
\label{sect-lorentz}

\subsection{The case $k=1$}

We will use the following two observations.

\begin{lemma}[see {\cite[Lemma~2.1]{KosZh}}]\label{lem-lorentz-equivalence}
Let $\alpha\in(0,1]$ and let $\nu$ be a probability Borel measure on $\mathbb R^k$.
Then the estimate
\begin{equation}\label{eq-lorentz-1}
\nu(A)\le C\bigl(\lambda_k(A)\bigr)^\alpha
\end{equation}
for every Borel set $A\subset \mathbb R^k$ is equivalent to the estimate
\begin{equation}\label{eq-lorentz-2}
\int_{\mathbb R^k}\psi\,d\nu\le C\|\psi\|_{L^1(\mathbb R^k)}^\alpha
\end{equation}
for every $\psi\in C_0^\infty(\mathbb R^k)$ with $0\le\psi\le 1$.
\end{lemma}

\begin{lemma}[see {\cite[Corollary~2.3]{KosZh}}]\label{lem-lorentz}
Let $k,d\in\mathbb N$, and let $\varrho$ be a nonnegative integrable function
on $\mathbb R^k$ such that
\[
0\le \varrho(x)\le M
\quad \forall x\in\mathbb R^k.
\]
Let $f\in \mathcal P_d(\mathbb R^k;\mathbb R^k)$. Then, for every
$\varepsilon>0$ and every $\psi\in C_0^\infty(\mathbb R^k)$ with
$0\le \psi\le 1$, one has
\[
\int_{\mathbb R^k}\psi(f(x))\varrho(x)\,dx
\le
\int_{\mathbb R^k}I_{\{|\det J_f|\le \varepsilon\}}\varrho\,dx
+
d^kM\varepsilon^{-1}\|\psi\|_{L^1(\mathbb R^k)}.
\]
\end{lemma}

\begin{theorem}\label{th-bounded-lorentz}
Let $M>0$ and let $d,n\in\mathbb N$. Let
$\mu=\mu_1\otimes\cdots\otimes\mu_n$,
where each $\mu_j$ is a probability measure with density $\varrho_j$ satisfying
\[
0\le \varrho_j(x_j)\le M
\quad \forall x_j\in\mathbb R,
\quad \forall j\in\{1,\ldots,n\}.
\]
Then, for every nonconstant $f\in \mathcal{P}_d(\mathbb{R}^n)$
and for every Borel set $A\subset\mathbb R$, one has
\[
[f]_\infty^{1/d}\mu(f\in A)
\le
2d^2(1+M)\bigl(\lambda_1(A)\bigr)^{1/d}.
\]
\end{theorem}

\begin{proof}
First we prove that, for every nonconstant polynomial $f$,
\begin{equation}\label{lorentz-1}
[f]_\infty^{1/d(f)}\mu(f\in A)
\le
2d(f)^2M\bigl(\lambda_1(A)\bigr)^{1/d(f)}
\end{equation}
for every Borel set $A\subset\mathbb R$.

Let
\[
f(x)=\sum_{|\mathbf m|\le d(f)}a_{\bf m}(f)x^{\mathbf m}.
\]
We prove \eqref{lorentz-1} by induction on $d(f)$.

First consider the case $d(f)=1$. Choose $j\in\{1,\ldots,n\}$ such that
\[
[f]_\infty=|a_{e_j}(f)|.
\]
For fixed $x_i$, $i\ne j$, using $\varrho_j\le M$ and changing variables in
the $x_j$-integral, we obtain
\[
\mu_j\bigl(x_j\in\mathbb R\colon f(x)\in A\bigr)
\le
M |a_{e_j}(f)|^{-1}\lambda_1(A)
=
M[f]_\infty^{-1}\lambda_1(A).
\]
Integrating with respect to the remaining variables gives
\[
\mu(f\in A)
\le
M[f]_\infty^{-1}\lambda_1(A)
\le
2M[f]_\infty^{-1}\lambda_1(A).
\]
Thus \eqref{lorentz-1} holds when $d(f)=1$.

Assume now that \eqref{lorentz-1} has been proved for all nonconstant
polynomials of degree at most $m-1$, where $m\ge2$, and let
$f\in \mathcal{P}_m(\mathbb R^n)$ satisfy $d(f)=m$. Choose a multi-index
${\bf r}\in\mathbb Z_+^n$ such that
\[
[f]_\infty=|a_{\bf r}(f)|,
\quad
|{\bf r}|=m.
\]
Without loss of generality, we may assume that $r_1\ge1$.

Fix $x_2,\ldots,x_n$. Applying Lemma~\ref{lem-lorentz} in the variable $x_1$
to the one-dimensional polynomial
\[
x_1\mapsto f(x_1,x_2,\ldots,x_n),
\]
we obtain, for every $\varepsilon>0$
and every $\psi\in C_0^\infty(\mathbb R)$ with $0\le\psi\le 1$, that
\[
\int_{\mathbb R}\psi(f(x))\,\mu_1(dx_1)
\le
\int_{\mathbb R}
I_{\{|\partial_{x_1}f(x)|\le \varepsilon\}}\,\mu_1(dx_1)
+
mM\varepsilon^{-1}\|\psi\|_{L^1(\mathbb R)}.
\]
Integrating with respect to $x_2,\ldots,x_n$, we obtain
\[
\int_{\mathbb R^n}\psi(f)\,d\mu
\le
\mu\bigl(|\partial_{x_1}f|\le \varepsilon\bigr)
+
mM\varepsilon^{-1}\|\psi\|_{L^1(\mathbb R)}.
\]

We note that $d(\partial_{x_1}f)=m-1$ and that
\[
[\partial_{x_1}f]_\infty\ge r_1|a_{\bf r}(f)|\ge [f]_\infty.
\]
By the induction hypothesis,
\[
\mu\bigl(|\partial_{x_1}f|\le\varepsilon\bigr)
=
\mu\bigl(\partial_{x_1}f\in[-\varepsilon,\varepsilon]\bigr)
\le
2(m-1)^2M[f]_\infty^{-\frac1{m-1}}
(2\varepsilon)^{\frac1{m-1}}.
\]
Therefore,
\[
\int_{\mathbb R^n}\psi(f)\,d\mu
\le
2(m-1)^2M[f]_\infty^{-\frac1{m-1}}(2\varepsilon)^{\frac1{m-1}}
+
mM\varepsilon^{-1}\|\psi\|_{L^1(\mathbb R)}.
\]
We now take
\[
\varepsilon=
\frac{1}{2}
m^{\frac{m-1}{m}}(m-1)^{-\frac{m-1}{m}}
[f]_\infty^{1/m}
\|\psi\|_{L^1(\mathbb R)}^{\frac{m-1}{m}}.
\]
Substituting this value of $\varepsilon$, we get
\[
\int_{\mathbb R^n}\psi(f)\,d\mu
\le
2m(m-1)M
m^{1/m}(m-1)^{-1/m}
[f]_\infty^{-1/m}
\|\psi\|_{L^1(\mathbb R)}^{1/m}
\le
2m^2M
[f]_\infty^{-1/m}
\|\psi\|_{L^1(\mathbb R)}^{1/m}
\]
for every $\psi\in C_0^\infty(\mathbb R)$ with $0\le\psi\le1$.
Applying Lemma~\ref{lem-lorentz-equivalence} to the image measure
$\mu\circ f^{-1}$ gives
\[
\mu(f\in A)
\le
2m^2M[f]_\infty^{-1/m}\bigl(\lambda_1(A)\bigr)^{1/m}
\]
for every Borel set $A\subset\mathbb R$. Since $m=d(f)$, this proves
\eqref{lorentz-1}.

Finally, let $f\in \mathcal P_d(\mathbb R^n)$ be nonconstant. Then
$1\le d(f)\le d$. If
$[f]_\infty^{-1}\lambda_1(A)\le 1$,
then, by \eqref{lorentz-1},
\[
\mu(f\in A)
\le
2d(f)^2M\bigl([f]_\infty^{-1}\lambda_1(A)\bigr)^{1/d(f)}
\le
2d^2(1+M)\bigl([f]_\infty^{-1}\lambda_1(A)\bigr)^{1/d}.
\]

If
$[f]_\infty^{-1}\lambda_1(A)>1$,
then
\[
\mu(f\in A)
\le
1
\le
\bigl([f]_\infty^{-1}\lambda_1(A)\bigr)^{1/d}
\le
2d^2(1+M)\bigl([f]_\infty^{-1}\lambda_1(A)\bigr)^{1/d}.
\]
Thus, in both cases,
\[
[f]_\infty^{1/d}\mu(f\in A)
\le
2d^2(1+M)\bigl(\lambda_1(A)\bigr)^{1/d},
\]
which completes the proof.
\end{proof}

In particular, we obtain the following counterpart of the classical
Carbery--Wright small-ball estimate \cite{CW01}.

\begin{corollary}\label{cor-bounded-CW}
Let $M>0$ and let $d,n\in\mathbb N$. Let
$\mu=\mu_1\otimes\cdots\otimes\mu_n$,
where each $\mu_j$ is a probability measure with density $\varrho_j$ satisfying
\[
0\le \varrho_j(x_j)\le M
\quad \forall x_j\in\mathbb R,
\quad \forall j\in\{1,\ldots,n\}.
\]
Then, for every $f\in \mathcal{P}_d(\mathbb{R}^n)$
and every $\varepsilon>0$, one has
\[
[f]_\infty^{1/d}\mu(|f|\le \varepsilon)
\le
4d^2(1+M)\varepsilon^{1/d}.
\]
\end{corollary}

\begin{proof}
For nonconstant $f$, this follows from Theorem~\ref{th-bounded-lorentz} applied to
$A=[-\varepsilon,\varepsilon]$. For constant $f$, the estimate is trivial.
\end{proof}

\begin{corollary}\label{cor-GM-type}
Let $M>0$ and let $d,n\in\mathbb N$. Let
$\mu=\mu_1\otimes\cdots\otimes\mu_n$,
where each $\mu_j$ is a probability measure with density $\varrho_j$ satisfying
\[
0\le \varrho_j(x_j)\le M
\quad \forall x_j\in\mathbb R,
\quad \forall j\in\{1,\ldots,n\}.
\]
Then, for every $f\in \mathcal{P}_d(\mathbb{R}^n)$, one has
\[
[f]_\infty
\le
2(8d^2(1+M))^d\|f\|_{L^1(\mu)}.
\]
\end{corollary}

\begin{proof}
The assertion is trivial for $f=0$. Suppose that $f\ne0$. By
Corollary~\ref{cor-bounded-CW}, we have
\[
[f]_\infty^{1/d}\mu(|f|\le \varepsilon)
\le
4d^2(1+M)\varepsilon^{1/d}
\quad \forall \varepsilon>0.
\]
By Markov's inequality,
\[
\mu\bigl(|f|\le 2\|f\|_{L^1(\mu)}\bigr)
=
1-\mu\bigl(|f|>2\|f\|_{L^1(\mu)}\bigr)
\ge
\frac12.
\]
Taking $\varepsilon=2\|f\|_{L^1(\mu)}$, we obtain
\[
[f]_\infty
\le
2(8d^2(1+M))^d\|f\|_{L^1(\mu)},
\]
which completes the proof.
\end{proof}

\subsection{Proof of Theorem~\ref{th-lorentz}}

We fix
\[
{\bf r}:=(r_1,\ldots,r_k)
\quad\text{with}\quad
1\le r_1<\ldots<r_k\le n.
\]
We write
\[
\mu_{\bf r}:=\mu_{r_1}\otimes\ldots\otimes\mu_{r_k}
\]
and
\[
x_{\bf r}:=(x_{r_1},\ldots,x_{r_k}).
\]
Let $x_{\hat{\bf r}}$ denote the remaining variables.

For every $\varepsilon>0$, every
$\psi\in C_0^\infty(\mathbb R^k)$ with $0\le \psi\le 1$, and every fixed
$x_{\hat{\bf r}}\in\mathbb R^{n-k}$, Lemma~\ref{lem-lorentz} gives
\[
\int_{\mathbb R^k}\psi(f(x_{\bf r},x_{\hat{\bf r}}))\,\mu_{\bf r}(dx_{\bf r})
\le
\int_{\mathbb R^k}I_{\{|\det J_f^{\bf r}|\le \varepsilon\}}\,\mu_{\bf r}(dx_{\bf r})
+
d^kM^k\varepsilon^{-1}\|\psi\|_{L^1(\mathbb R^k)}.
\]
Integrating with respect to the remaining variables gives
\begin{equation}\label{eq-lorentz-pre-CW}
\int_{\mathbb R^n}\psi(f)\,d\mu
\le
\mu\bigl(|\det J_f^{\bf r}|\le \varepsilon\bigr)
+
d^kM^k\varepsilon^{-1}\|\psi\|_{L^1(\mathbb R^k)}.
\end{equation}

Assume that 
$[\det J_f^{\bf r}]_\infty>0$.
Since $\det J_f^{\bf r}\in\mathcal P_{k(d-1)}(\mathbb R^n)$, Corollary~\ref{cor-bounded-CW}
gives
\[
\mu\bigl(|\det J_f^{\bf r}|\le \varepsilon\bigr)
\le
4k^2d^2(1+M)
[\det J_f^{\bf r}]_\infty^{-\frac{1}{k(d-1)}}
\varepsilon^{\frac{1}{k(d-1)}}.
\]
Combining this with \eqref{eq-lorentz-pre-CW} and taking
\[
\varepsilon
=
\biggl(
\frac{
d^kM^k\|\psi\|_{L^1(\mathbb R^k)}
[\det J_f^{\bf r}]_\infty^{\frac{1}{k(d-1)}}}
{k^2d^2(1+M)}
\biggr)^{\frac{k(d-1)}{k(d-1)+1}},
\]
we obtain
\[
\int_{\mathbb R^n}\psi(f)\,d\mu
\le
5(1+M)^{\frac{k(d-1)}{k(d-1)+1}}
(d^kM^k)^{\frac{1}{k(d-1)+1}}
(kd)^{\frac{2k(d-1)}{k(d-1)+1}}
[\det J_f^{\bf r}]_\infty^{-\frac{1}{k(d-1)+1}}
\|\psi\|_{L^1(\mathbb R^k)}^{\frac{1}{k(d-1)+1}}.
\]
Since $d\ge2$,
\[
\frac{1}{k(d-1)+1}\le \frac{2}{kd}.
\]
Hence
\[
[\det J_f^{\bf r}]_\infty^{\frac{1}{k(d-1)+1}}
\int_{\mathbb R^n}\psi(f)\,d\mu
\le
20(1+M)^2k^2d^2
\|\psi\|_{L^1(\mathbb R^k)}^{\frac{1}{k(d-1)+1}}.
\]

Choosing ${\bf r}$ so that
\[
[\det J_f^{\bf r}]_\infty
=
\max_{{\bf r}\in {\bf N}_n^k}[\det J_f^{\bf r}]_\infty
\]
and using
\[
\max_{{\bf r}\in {\bf N}_n^k}[\det J_f^{\bf r}]_\infty\ge a,
\]
we obtain
\[
\int_{\mathbb R^n}\psi(f)\,d\mu
\le
20k^2d^2(1+M)^2
a^{-\frac{1}{k(d-1)+1}}
\|\psi\|_{L^1(\mathbb R^k)}^{\frac{1}{k(d-1)+1}}.
\]

Lemma~\ref{lem-lorentz-equivalence} now gives the announced estimates.
\qed

\section{Proofs of the main results}
\label{sect-main}

\subsection{Comparison of coefficient normalizations}

\begin{lemma}\label{lem-coeff}
Let $M,R>0$ and $d\in\mathbb N$. Let
$\mu=\bigotimes_{j=1}^\infty\mu_j$,
where $\mu_j\in\mathcal M_d(M,R)$ for every $j\in\mathbb N$.
Then there exists a constant $C=C(d)>0$ such that, for every
$f=(f_1,f_2)$, where $f_1$ and $f_2$ are cylindrical polynomials of degrees
at most $d$, one has
\[
(1+R)^{-2d}[[f]]_{\mu,\infty}
\le
[[f]]_\infty
\le
C(d)(1+M)^{2d}[[f]]_{\mu,\infty}.
\]
\end{lemma}

\begin{proof}
Recall that, for every $j\in\mathbb N$,
$u_{j,0}=1,u_{j,1},\ldots,u_{j,d}$ are the orthonormal polynomials in
$L^2(\mu_j)$. Write
\[
u_{j,m}(t)=\sum_{k=0}^m\alpha_{j,k}^m t^k.
\]
We first estimate $\alpha_{j,m}^m$ for $0\le m\le d$. Clearly,
$\alpha_{j,0}^0=1$. For $m\ge1$, since $u_{j,m}$ is orthogonal in
$L^2(\mu_j)$ to
\[
\sum_{k=0}^{m-1}\alpha_{j,k}^m t^k,
\]
we obtain
\[
1=\|u_{j,m}\|_{L^2(\mu_j)}
\le
\|\alpha_{j,m}^m t^m\|_{L^2(\mu_j)}
\le
|\alpha_{j,m}^m|R^m.
\]
Hence
\begin{equation}\label{eq-upper}
|\alpha_{j,m}^m|
\ge
R^{-m}
\ge
(1+R)^{-m}.
\end{equation}
On the other hand, applying Corollary~\ref{cor-GM-type} to $u_{j,m}$ and
using $\|u_{j,m}\|_{L^1(\mu_j)}\le1$, we obtain
\begin{equation}\label{eq-lower}
|\alpha_{j,m}^m|
\le
2\bigl(8m^2(1+M)\bigr)^m
\le
(16d^2)^d(1+M)^m.
\end{equation}

Now let $f=(f_1,f_2)$, where $f_1$ and $f_2$ depend only on the first $n$
coordinates. Write
\[
f_i(x)
=
\sum_{|\mathbf m|\le d(f)}a_{\bf m}(f_i)x^{\mathbf m}
=
\sum_{|\mathbf m|\le d_\mu(f)}a_{\bf m}^\mu(f_i)U_{\mathbf m}(x),
\quad
i\in\{1, 2\},
\]
where
\[
U_{\mathbf m}(x)
=
\prod_{j=1}^n u_{j,m_j}(x_j).
\]
Since $U_{\mathbf m}$ has degree $|\mathbf m|$ and leading monomial
coefficient
\[
\prod_{j=1}^n\alpha_{j,m_j}^{m_j},
\]
the triangularity of the change of basis gives $d(f)=d_\mu(f)$ and
\[
a_{\mathbf m}(f_i)
=
\biggl(\prod_{j=1}^n\alpha_{j,m_j}^{m_j}\biggr)
a_{\bf m}^\mu(f_i)
\]
for every $|\mathbf m|=d(f)$ and $i\in\{1,2\}$. Consequently, for every
$\mathbf p,\mathbf q$ with
$|\mathbf p|=|\mathbf q|=d(f)$, we have
\[
A_{\mathbf p,\mathbf q}(f)
=
\biggl(\prod_{j=1}^n\alpha_{j,p_j}^{p_j}\biggr)
\biggl(\prod_{j=1}^n\alpha_{j,q_j}^{q_j}\biggr)
A^\mu_{\mathbf p,\mathbf q}(f).
\]
By \eqref{eq-upper} and \eqref{eq-lower},
\[
(1+R)^{-2d}
\le
\biggl(\prod_{j=1}^n|\alpha_{j,p_j}^{p_j}|\biggr)
\biggl(\prod_{j=1}^n|\alpha_{j,q_j}^{q_j}|\biggr)
\le
(16d^2)^{2d}(1+M)^{2d}.
\]
Taking the maximum over $\mathbf p$ and $\mathbf q$ proves the claim.
\end{proof}

\subsection{Proof of Theorem~\ref{th-inf-lorentz}}

Let
\[
f_1(x)=\sum_{|\mathbf m|\le d}a_{\bf m}^\mu(f_1)U_{\mathbf m}(x),
\qquad
f_2(x)=\sum_{|\mathbf m|\le d}a_{\bf m}^\mu(f_2)U_{\mathbf m}(x).
\]	
Without loss of generality, we may assume that
$[[f]]_{\mu,\infty}>0$.
By the definition of $[[f]]_{\mu,\infty}$, there exist
multi-indices $\mathbf p,\mathbf q$ such that
$|\mathbf p|=|\mathbf q|=d_\mu(f)$
and
\[
|A^\mu_{\mathbf p,\mathbf q}(f)|\ge \frac{1}{2}\, [[f]]_{\mu,\infty}.
\]
Choose $N_0\in\mathbb N$ such that
\[
\operatorname{supp}\mathbf p\cup\operatorname{supp}\mathbf q
\subset\{1,\ldots,N_0\}.
\]
For $N\ge N_0$, define
\begin{equation}\label{eq-trunc}
f_1^{(N)}(x)
:=
\sum_{\substack{|\mathbf m|\le d\\
\operatorname{supp}\mathbf m\subset\{1,\ldots,N\}}}
a_{\bf m}^\mu(f_1)U_{\mathbf m}(x),
\quad
f_2^{(N)}(x)
:=
\sum_{\substack{|\mathbf m|\le d\\
\operatorname{supp}\mathbf m\subset\{1,\ldots,N\}}}
a_{\bf m}^\mu(f_2)U_{\mathbf m}(x)
\end{equation}
and set
$f^{(N)}=(f_1^{(N)},f_2^{(N)})$.
Then
\[
f_1^{(N)}\to f_1
\quad\text{and}\quad
f_2^{(N)}\to f_2
\quad\text{in }L^2(\mu).
\]

For every $N\ge N_0$, 
by Lemma~\ref{lem-coeff},
\[
[[f^{(N)}]]_\infty
\ge
(1+R)^{-2d}[[f^{(N)}]]_{\mu, \infty}
\ge 
(1+R)^{-2d}|A^\mu_{\mathbf p,\mathbf q}(f)|
\ge
(1+R)^{-2d}\, \frac{1}{2}\, [[f]]_{\mu,\infty}.
\]
By Theorem~\ref{th-algebr},
\[
\max_{{\bf r}\in{\bf N}_N^2}
[\det J_{f^{(N)}}^{\bf r}]_\infty
\ge
C_1(d)(1+R)^{-2d}[[f]]_{\mu,\infty}.
\]
Applying Theorem~\ref{th-lorentz} with $k=2$ and
Lemma~\ref{lem-lorentz-equivalence}, we obtain, for every
$\psi\in C_0^\infty(\mathbb R^2)$ with $0\le\psi\le1$,
\begin{align*}
\int_{\mathbb R^\infty}\psi(f^{(N)})\,d\mu
&=
\int_{\mathbb R^N}\psi\bigl(f^{(N)}(x_1, \ldots, x_N)\bigr)\,\mu_1(dx_1)\ldots\mu_N(dx_N)
\\
&\le
80d^2(1+M)^2
\bigl(C_1(d)(1+R)^{-2d}[[f]]_{\mu,\infty}\bigr)^{-\frac1{2d-1}}
\|\psi\|_{L^1(\mathbb R^2)}^{\frac1{2d-1}}.
\end{align*}
Passing to the limit, we obtain
\[
\int_{\mathbb R^\infty}\psi(f)\,d\mu
\le
C_2(d)(1+M)^2
(1+R)^{\frac{2d}{2d-1}}[[f]]_{\mu,\infty}^{-\frac1{2d-1}}
\|\psi\|_{L^1(\mathbb R^2)}^{\frac1{2d-1}},
\]
where
\[
C_2(d)=80d^2
\bigl(C_1(d)\bigr)^{-\frac1{2d-1}}.
\]
Applying Lemma~\ref{lem-lorentz-equivalence} to the image measure
$\mu\circ f^{-1}$ gives the claimed estimate.
\qed

\subsection{Markov--Bernstein-type inequality}

\begin{lemma}\label{lem-MB}
Let $M,R>0$ and let $d\in\mathbb N$.
There exists a constant $C(M,R,d)>0$ such that, for every
$\mu=\bigotimes_{j=1}^{\infty}\mu_j$
with $\mu_j\in\mathcal M_d(M,R)$ and every 
cylindrical polynomial $f$ of degree at most $d$,
one has
\[
\|\nabla f\|_{L^2(\mu)}
\le C(M,R,d)\|f\|_{L^2(\mu)}.
\]
\end{lemma}

\begin{proof}
We start with a one-dimensional estimate.
Fix $j\in\mathbb N$ and let
\[
g(t)=\sum_{k=0}^m a_kt^k,
\quad m\le d.
\]
Set
\[
g_r(t)=\sum_{k=0}^r a_kt^k,
\quad r\in\{0,\ldots,m\}.
\]
For every $r\in\{1,\ldots,m\}$, Corollary~\ref{cor-GM-type} gives
\[
|a_r|
\le
2(8r^2(1+M))^r\|g_r\|_{L^2(\mu_j)}
\le
C_1(M, d)\|g_r\|_{L^2(\mu_j)}.
\]
By this estimate and the moment assumption,
\[
\|g_{r-1}\|_{L^2(\mu_j)}
\le
\|g_r\|_{L^2(\mu_j)}
+
|a_r|\|t^r\|_{L^2(\mu_j)}
\le
C_2(M,R,d)\|g_r\|_{L^2(\mu_j)}.
\]
Iterating from $r=m$ down to $r=1$ gives
\[
\|g_r\|_{L^2(\mu_j)}
\le
C_2(M,R,d)^d\|g\|_{L^2(\mu_j)}
\quad \forall r\in\{0,\ldots,m\}.
\]
Consequently,
\[
\sum_{k=0}^m|a_k|
\le
C_3(M,R,d)\|g\|_{L^2(\mu_j)}.
\]
Using the moment assumption again, we obtain
\begin{equation}\label{eq-derivative-est}
\|g'\|_{L^2(\mu_j)}
\le
\sum_{k=1}^m k|a_k|\|t^{k-1}\|_{L^2(\mu_j)}
\le
C_4(M,R,d)\|g\|_{L^2(\mu_j)}.
\end{equation}

Recall that, for every $j\in\mathbb N$,
$u_{j,0}=1,u_{j,1},\ldots,u_{j,d}$
are the orthonormal polynomials in $L^2(\mu_j)$, and
\[
U_{\mathbf m}(x)
=
\prod_j u_{j,m_j}(x_j).
\]
Let now $f$ be a cylindrical polynomial depending only on the first $n$
coordinates and write
\[
f=
\sum_{\substack{|\mathbf m|\le d\\
\operatorname{supp}{\bf m}\subset\{1,\ldots,n\}}}
a_{\mathbf m}(f)U_{\mathbf m}.
\]

Fix $j\in\{1,\ldots,n\}$. We write
\[
f(x)
=
\sum_{|\hat{{\bf m}}_j|\le d}
f_{\hat{{\bf m}}_j}(x_j)
U_{\hat{{\bf m}}_j}(\hat{x}_j),
\]
where
\[
f_{\hat{{\bf m}}_j}(x_j)
:=
\sum_{m_j=0}^{d-|\hat{{\bf m}}_j|}
a_{(\hat{{\bf m}}_j,m_j)}(f)u_{j,m_j}(x_j),
\quad
\text{and}
\quad
U_{\hat{{\bf m}}_j}(\hat{x}_j)
:=
\prod_{k\ne j}u_{k,m_k}(x_k).
\]
Here
\[
\hat{x}_j
:=
(x_1,\ldots,x_{j-1},x_{j+1},\ldots,x_n),
\quad
\hat{{\bf m}}_j
:=
(m_1,\ldots,m_{j-1},m_{j+1},\ldots,m_n),
\]
and
\[
(\hat{{\bf m}}_j,m_j)
:=
(m_1,\ldots,m_{j-1},m_j,m_{j+1},\ldots,m_n).
\]

Applying \eqref{eq-derivative-est} to
$f_{\hat{{\bf m}}_j}-a_{(\hat{{\bf m}}_j,0)}(f)$
and using orthogonality, we obtain
\[
\|\partial_{x_j}f\|_{L^2(\mu)}^2
=
\sum_{|\hat{{\bf m}}_j|\le d}
\|f_{\hat{{\bf m}}_j}'\|_{L^2(\mu_j)}^2
\le
C_4(M,R,d)^2
\sum_{\mathbf m:\,m_j>0}|a_{\mathbf m}(f)|^2.
\]
Therefore,
\begin{align*}
\|\nabla f\|_{L^2(\mu)}^2
=
\sum_{j=1}^n\|\partial_{x_j}f\|_{L^2(\mu)}^2
&\le
C_4(M,R,d)^2
\sum_{|\mathbf m|\le d}
\#\{j\colon m_j>0\}|a_{\mathbf m}(f)|^2
\\
&\le
dC_4(M,R,d)^2
\sum_{|\mathbf m|\le d}|a_{\mathbf m}(f)|^2
=
dC_4(M,R,d)^2\|f\|_{L^2(\mu)}^2,
\end{align*}
which completes the proof.
\end{proof}

\begin{corollary}\label{cor-grad}
Let $M,R>0$ and $d\in\mathbb N$. Let
$\mu=\bigotimes_{j=1}^\infty\mu_j$,
where $\mu_j\in\mathcal M_d(M,R)$ for every $j\in\mathbb N$.
Then, for every $f\in\mathcal P_d(\mu)$ and every sequence of cylindrical
polynomials $\{f_n\}_{n=1}^\infty$ of degree at most $d$ such that
$f_n\to f$ in $L^2(\mu)$, the sequence $\{\nabla f_n\}_{n=1}^\infty$
converges in $L^2(\mu;\ell^2)$, and its limit does not depend on the choice
of $\{f_n\}_{n=1}^\infty$. Consequently, the gradient extends uniquely to a
bounded linear operator
$\nabla\colon\mathcal P_d(\mu)\to L^2(\mu;\ell^2)$.
\end{corollary}

\subsection{Proof of Theorem~\ref{th-det-est}}

By Corollary~\ref{cor-grad}, the gradient is well defined on
$\mathcal P_d(\mu)$.

First, let $f=(f_1,f_2)$, where $f_1$ and $f_2$ are cylindrical polynomials
of degree at most $d$ depending only on the first $n$ coordinates. 
If $f$ is constant, the conclusion is immediate. Otherwise, by
Theorem~\ref{th-algebr}, there exists ${\bf r}_0\in{\bf N}_n^2$ such that
\[
[\det J_f^{{\bf r}_0}]_\infty
\ge
C(d)[[f]]_\infty.
\]
By Lemma~\ref{lem-coeff},
\[
[[f]]_\infty
\ge
(1+R)^{-2d}[[f]]_{\mu,\infty}.
\]
Therefore,
\[
[\det J_f^{{\bf r}_0}]_\infty
\ge
C_1(R,d)[[f]]_{\mu,\infty}.
\]
By Corollary~\ref{cor-GM-type},
\[
\|\det J_f^{{\bf r}_0}\|_{L^1(\mu)}
\ge
C_2(M,d)[\det J_f^{{\bf r}_0}]_\infty
\ge
C_3(M,R,d)[[f]]_{\mu,\infty}.
\]
By \eqref{eq-decomp},
\begin{equation}\label{eq-fin-dim}
	\int_{\mathbb R^\infty}\Delta_f^{1/2}\,d\mu
	\ge
	\|\det J_f^{{\bf r}_0}\|_{L^1(\mu)}
	\ge
	C_3(M,R,d)[[f]]_{\mu,\infty}.
\end{equation}

Now let $f=(f_1,f_2)$ with $f_1,f_2\in\mathcal P_d(\mu)$, and write
\[
f_1(x)=\sum_{|\mathbf m|\le d}a_{\bf m}^\mu(f_1)U_{\mathbf m}(x),
\qquad
f_2(x)=\sum_{|\mathbf m|\le d}a_{\bf m}^\mu(f_2)U_{\mathbf m}(x).
\]
If $[[f]]_{\mu,\infty}=0$, there is nothing to prove. Otherwise, we choose
multi-indices $\mathbf p$ and $\mathbf q$ with
$|\mathbf p|=|\mathbf q|=d_\mu(f)$ such that
\[
|a_{\bf p}^\mu(f_1)a_{\bf q}^\mu(f_2)
-
a_{\bf q}^\mu(f_1)a_{\bf p}^\mu(f_2)|
\ge
2^{-1}[[f]]_{\mu,\infty},
\]
and choose $N_0$ such that
\[
\operatorname{supp}\mathbf p\cup\operatorname{supp}\mathbf q
\subset
\{1,\ldots,N_0\}.
\]
For $N\ge N_0$, define
$f^{(N)}=(f_1^{(N)},f_2^{(N)})$ as in \eqref{eq-trunc}. Then
\[
[[f^{(N)}]]_{\mu,\infty}
\ge
2^{-1}[[f]]_{\mu,\infty},
\]
and \eqref{eq-fin-dim} gives
\begin{equation}\label{eq-trunc-est}
	\int_{\mathbb R^\infty}\Delta_{f^{(N)}}^{1/2}\,d\mu
	\ge
	C_4(M,R,d)[[f]]_{\mu,\infty}.
\end{equation}

Since $f_i^{(N)}\to f_i$ in $L^2(\mu)$ for $i\in\{1,2\}$,
Corollary~\ref{cor-grad} yields
\[
\|\nabla f_i-\nabla f_i^{(N)}\|_{L^2(\mu)}\to 0
\quad \forall i\in\{1,2\}.
\]
For $u,v\in\ell^2$, $u\ne 0$, we have
\[
|u|^2|v|^2-\langle u,v\rangle^2
=
|u|^2
\bigl|
v-\bigl\langle\tfrac{u}{|u|},v\bigr\rangle\tfrac{u}{|u|}
\bigr|^2.
\]
Therefore,
\[
\bigl|
\bigl(|u|^2|v|^2-\langle u,v\rangle^2\bigr)^{1/2}
-
\bigl(|u|^2|w|^2-\langle u,w\rangle^2\bigr)^{1/2}
\bigr|
\le
|u||v-w|
\]
for all $u,v,w\in\ell^2$. Hence,
\[
|\Delta_f^{1/2}-\Delta_{f^{(N)}}^{1/2}|
\le
|\nabla f_1||\nabla f_2-\nabla f_2^{(N)}|
+
|\nabla f_1-\nabla f_1^{(N)}||\nabla f_2^{(N)}|.
\]
Thus, by the Cauchy--Schwarz inequality,
\[
\|\Delta_f^{1/2}-\Delta_{f^{(N)}}^{1/2}\|_{L^1(\mu)}
\le
\|\nabla f_1\|_{L^2(\mu)}
\|\nabla f_2-\nabla f_2^{(N)}\|_{L^2(\mu)}
+
\|\nabla f_1-\nabla f_1^{(N)}\|_{L^2(\mu)}
\|\nabla f_2^{(N)}\|_{L^2(\mu)}.
\]
Therefore, passing to the limit
in \eqref{eq-trunc-est}, we obtain
\[
\int_{\mathbb R^\infty}\Delta_f\,d\mu
\ge
\biggl(\int_{\mathbb R^\infty}\Delta_f^{1/2}\,d\mu\biggr)^2
\ge
C_4(M,R,d)^2[[f]]_{\mu,\infty}^2,
\]
which concludes the proof.
\qed

\section*{Use of AI Tools}

ChatGPT was used for language editing, stylistic suggestions, and draft wording for
selected passages. All AI-generated
text was checked, corrected where necessary, and
substantially revised by the author. The author take full responsibility for
the content of the paper.

\section*{Acknowledgements}

The author was supported by the AEI grants 
RYC2023-043616-I and
PID2025-169712NA-I00 funded by MICIU/AEI/10.13039/501100011033,
and by the Spanish State Research Agency, through the Severo Ochoa and Mar\'ia de Maeztu Program for Centers and
Units of Excellence in R\&D (CEX2020-001084-M).
The author thanks CERCA Programme (Generalitat de Catalunya) for institutional support.
{\sloppy
	
}

\end{document}